\documentclass[english, 11pt]{amsart}

\usepackage{latexsym}
\usepackage{color}

\usepackage{tikz-cd}
\usepackage[all]{xy}
\usepackage[utf8x]{inputenc}  
\usepackage[english]{babel}
\usepackage{mathtools}
\usepackage[T1]{fontenc}
\usepackage{enumitem}

\usepackage{graphicx,epsfig,amscd,graphics,subcaption}
\usepackage{amsmath,amssymb}

\newcommand{\bigslant}[2]{{\raisebox{.2em}{$#1$}\left/\raisebox{-.2em}{$#2$}\right.}}

\newcommand{\windtree}{\mathcal{W}}

\newcommand{\torusl}{\mathcal{T}_\Lambda}

\newcommand{\C}{\mathbb{C}}
\newcommand{\Z}{\mathbb{Z}}
\newcommand{\N}{\mathbb{N}}
\newcommand{\R}{\mathbb{R}}

\newcommand{\GL}{\mathrm{GL}}
\newcommand{\SL}{\mathrm{SL}}
\newcommand{\SO}{\mathrm{SO}}
\newcommand{\Gr}{\mathrm{Gr}}

\newcommand{\quadra}{\mathcal{Q}}

\newcommand{\isoC}{\operatorname{Iso}^+ (\C)}

\newcommand{\bilfam}{\mathcal {B}}

\newcommand{\re}{\mathrm{Re}}
\newcommand{\im}{\mathrm{Im}}

\newtheorem{theorem}{Theorem}
\newtheorem*{theorem*}{Theorem}
\newtheorem{corollary}{Corollary}
\newtheorem{lemma}{Lemma}
\newtheorem*{lemma*}{Lemma}
\newtheorem{proposition}{Proposition}

\theoremstyle{definition}
\newtheorem{definition}{Definition}

\begin{document}
\title{Diffusion rate of windtree models and Lyapunov exponents}
\author{Charles Fougeron}
\maketitle

\begin{abstract}
	Consider a windtree model with several parallel arbitrary right-angled obstacles
	placed periodically on the plane. We show that its diffusion
	rate is the largest Lyapunov exponent of some stratum of quadratic 
	differentials and exhibit a new general strategy to compute the
	generic diffusion rate of such models. This result enables us to compute
	numerically the diffusion rates of a large family of models and to
	observe its asymptotic behaviour according to the shape of the
	obstacles.
\end{abstract}


\section{Introduction.}

The windtree model was first introduced by Paul and
Tatiana Ehrenfest in 1912 \cite{Ehr} as part of statistical physics
investigations. In this book they set a simplified model for non interacting light
particles moving around massive particles that do not move but on which
the light particles bounce with elastic collision. We classically refer to the
light particles as the \textit{wind} and the static ones as \textit{trees}. The
motivation of the two physicists was to understand the kinetic behaviour of
such a system. They asked, among others, the following question: \textit{for a
	\text{generic} disposition of square trees orientated in the same
	direction, does the speed of $K$ light particles equidistributes
asymptotically in the $4$ possible directions ?} 

Plenty of questions have been studied on this model, in particular for the
$\Z^2$-periodic case with square obstacles. The results feature alternatively
elements of chaotic and periodic behaviour. In \cite{HardyWeber} was proven on
the one hand recurrence of the billiard flow and on the other hand abnormal
diffusion for special dimensions of the obstacles, \cite{Ulci} showed
genericity of non-ergodic behaviour, and its diffusion rate was computed to be
$2/3$ in \cite{DHL}. A positive answer to the original question has only been provided
very recently by \cite{Alba2}. 

In parallel a similar model with smooth convex obstacles has been studied by a
large amount of mathematicians throughout the twentieth century (see e.g.
\cite{SinaiBu} or \cite{SV}). In this case, the billiards satisfy some
hyperbolicity property and the behaviour of its flow is closely related to a
Brownian motion. 

A good tool to check if a polygonal windtree model has such an hyperbolic
behaviour is provided by the diffusion rates which should be $1/2$ in the case
of Brownian-like motions. In particular the result of \cite{DHL} breaks any
hope to apply directly methods of the smooth convex case to the rectangular
model. The question is still open in the case of asymptotic of polygonal shapes
approaching smooth convex ones, for example with the circle : is the diffusion
rate of periodic windtree models with regular $n$-gons going to $1/2$ when $n$
goes to $\infty$ ? We hope that developing methods to compute these diffusion
rates in more general settings provide a first step to understanding this
asymptotic and the non-convex obstacles cases.\\

The arguments of \cite{DHL} relies on a remarkable correspondence between the
diffusion rate of an infinite periodic billiard table and the Lyapunov exponent
of an associated translation surface.  This computation was generalised in
\cite{DZ} to any $\Z^2$-periodic windtree which trees have only right angles
and are horizontally and vertically symmetric.  In every of these cases, the
corresponding Lyapunov exponent belongs to some $2$ dimensional subbundle
of the Hodge bundle. Moreover in all of these cases the Lyapunov exponent is
rational and can be computed using some geometric arguments.\\

In this article we describe a general strategy to exhibit the Lyapunov
exponent of some locus in a stratum that corresponds to the diffusion rate of
a given periodic windtree model. It relies on two main ingredients : the first
one is to identify a common orbit closure of almost all translation
surfaces associated to a family of windtree tables; the second one is to find an
irreducible subbundle of the Hodge bundle on this locus which top Lyapunov
exponent is exactly the diffusion rate.  The tools for the first craft are
given by recent results of \cite{EMM}, \cite{Wright14} and \cite{Wright13} and
are introduced in subsection \ref{closure}. For the second one, we show an
additional lemma to the work of \cite{EC} which yields the diffusion rate for
any translation surface in a generic direction.\\

We apply this method to the case of periodic windtree with several obstacles in
its fundamental domain. Pick a family of $n \geq 2$ rectangular obstacles in a
square, and repeat this table $\Z^2$-periodically in the plane. We show the
following theorem,

\begin{theorem*}
	The diffusion rate for almost every such windtree model, in almost
	every direction is equal to the top Lyapunov exponent of $\quadra
	(1^{4n})$.
\begin{figure}[h!]
  \centering
  \includegraphics[width=7cm]{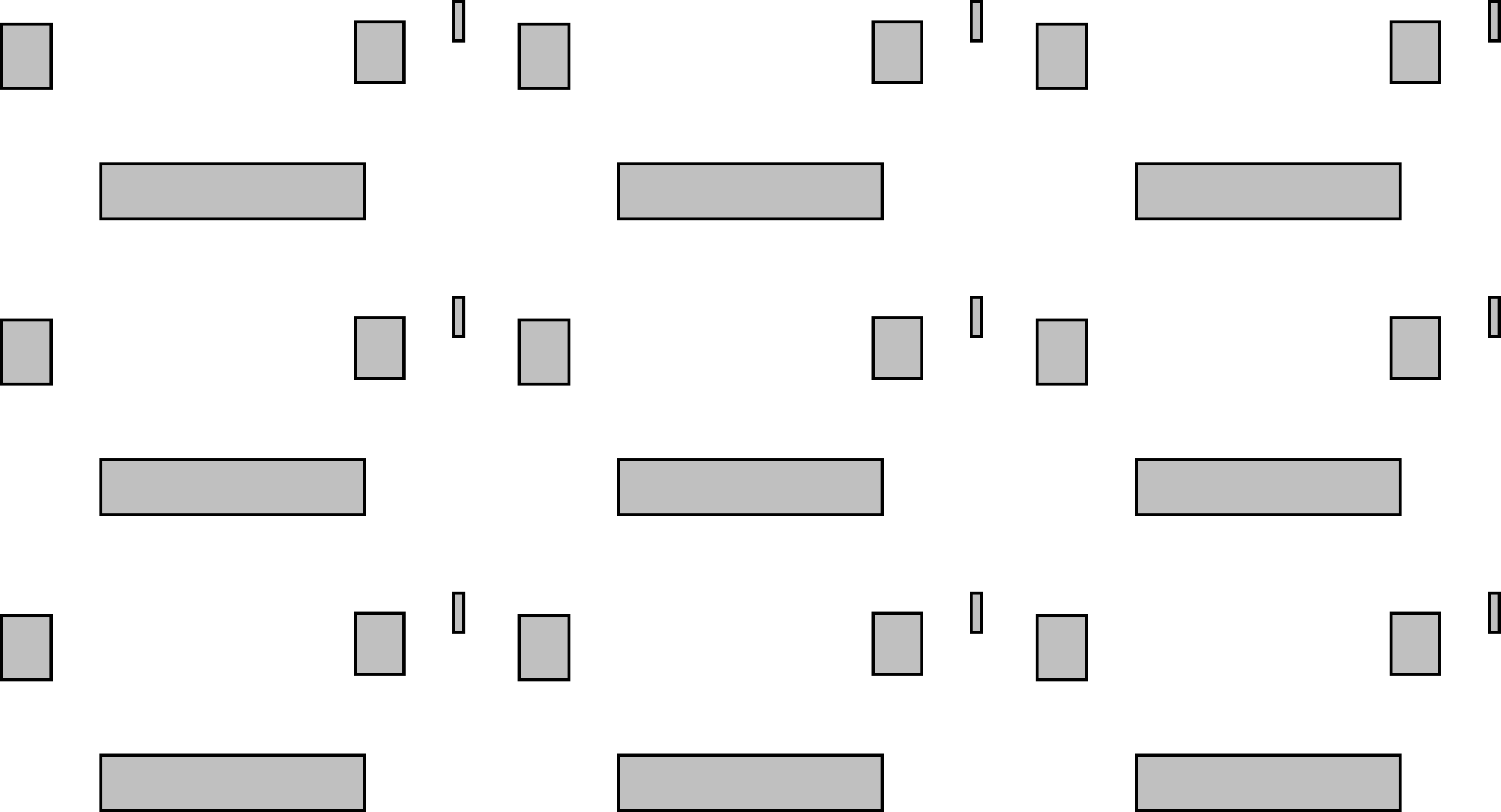}
\end{figure}
\end{theorem*}

Now take more general obstacles again with right angles, if $n$ is the number
of obstacles ans $p$ the total number of inward (concave) right angles in all the
obstacles, we have a similar result,
\begin{theorem*}
	The diffusion rate for almost every such windtree model, in almost
	every direction is equal to the top Lyapunov exponent of $\quadra
	(1^{4n + p}, -1^p)$.
	\begin{figure}[h] \centering
		\includegraphics[width=7cm]{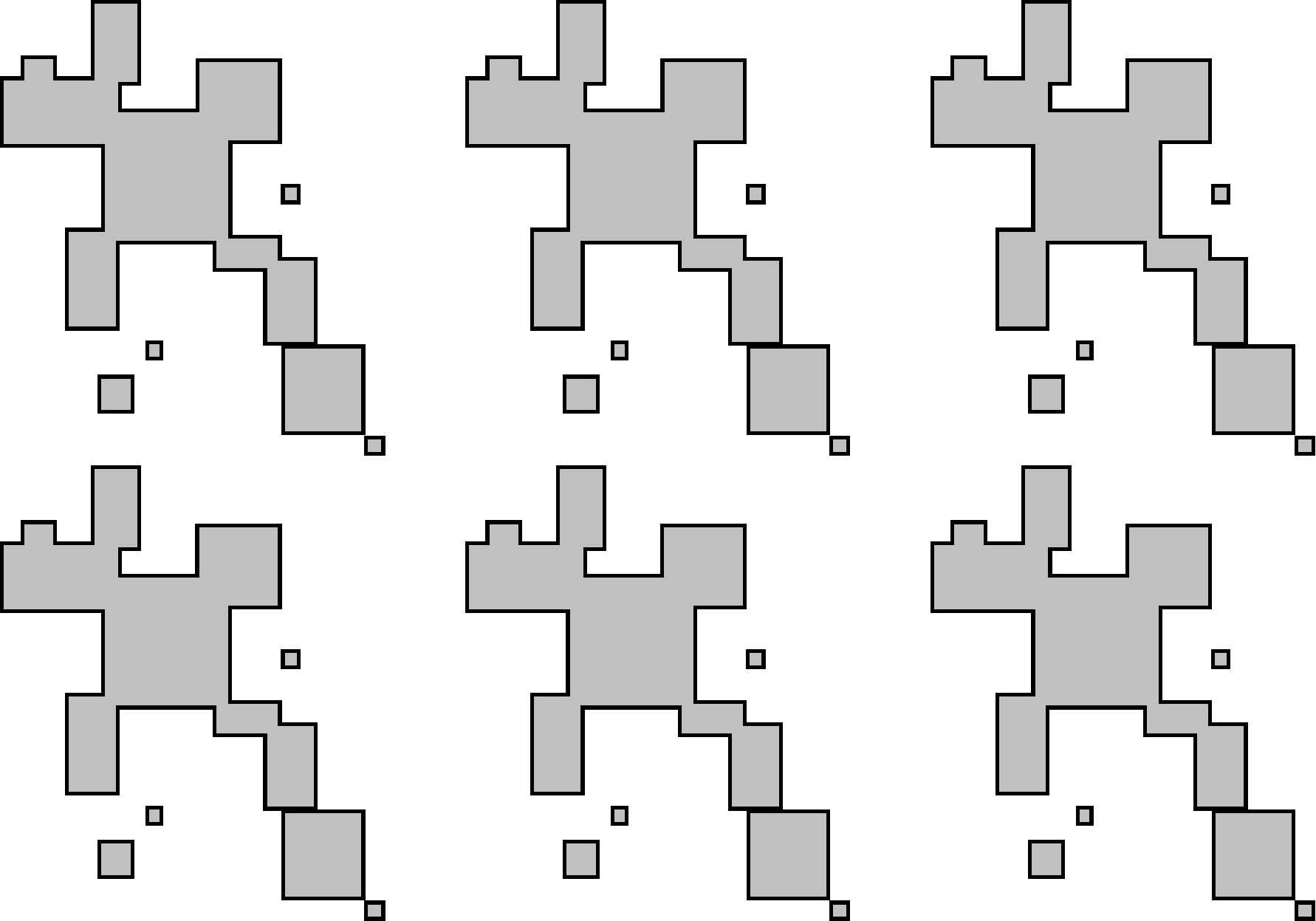}
	\end{figure}
\end{theorem*}

In the last section we discuss the value of these exponents running numerical
experiments with a Sage code developed by the author in a collaborative project
\cite{flatsurf}. These experiments give strong evidences that the family we
have introduced above can approach arbitrarily close any diffusion rate between
$1/2$ and $0$. In particular it goes to $1/2$ (\textit{i.e.} the diffusion rate
of the Brownian motion) when the number of obstacles goes to infinity.

\section{Translation surfaces}

\subsection{Definition}

A \textit{translation surface} is a surface whose change of charts are
translations.  Such a surface is endowed with a flat metric (the pull-back of
the canonical metric on $\R^2$) and a canonical direction.\\

One way to think of these translation surfaces is by gluing sides of a polygon
via translations.  Let $P$ be a polygon with $2k$ edges and let $z_1, \dots,
z_{2k}$ be complex numbers associated to the vectors of its sides.  Assume that
$z_i = z_{k+i}$, we glue the sides $z_i$ and $z_{k+i}$ and obtain a flat
surface with conical singularities of  angle multiples of $2\pi$. \\

We can define similar structures allowing the change of charts to be also
translations composed with $-\mathrm{Id}$. The class of surfaces we obtain are
called \textit{half-translation} surfaces.\\

Using triangulations Veech showed in \cite{Veech93} that this is a general
construction with a notion of \textit{pseudo-polygons} (in a much wider class
of structures). The complex numbers  $(z_i)_{1 \leq i \leq k}$ (defined up to a
sign in the case of half-translation surfaces) induce local coordinates in the
moduli space of such structures, we call them \textit{period coordinates}.  We
will introduce them as periods of abelian differentials below.

\subsubsection{Differentials and Moduli spaces}

There is a one-to-one correspondence between compact translation surfaces and
Riemann surfaces equipped with a non-zero holomorphic $1$-form. As well as
between compact half-translation surfaces and Riemann surfaces equipped with
quadratic differentials.\\

For $g \geq 1$ let $\alpha$ and $\beta$ be partitions of $2g-2$ and $4g-4$. The
strata $\mathcal H(\alpha)$ and $\quadra (\beta)$ are defined to be the sets of
$(S, \omega)$ and $(S, q)$ where $S$ is a genus $g$ closed Riemann surface,
$\omega$ is a holomorphic $1$-form on $S$, $q$ is a quadratic differential on
$S$, and their zeros multiplicities are given by $\alpha$ and $\beta$.  The
conical points in a translation surface correspond to the zeros of the
differential. If $d$ is the multiplicity of the zero, the angle is equal to
$2(d+1)\pi$ (and $(d+2)\pi$ for half-translation surfaces).\\

Given a translation surface $(S, \omega)$, let $\Sigma \subset S$ be the set of
zeros of $\omega$. Pick a basis $\{ \xi_1, \dots, \xi_n \}$ for the relative
homology group $H_1(S, \Sigma; \Z)$. The map $\Phi: \mathcal H(\alpha) \to
\C^n$ defined by
$$\Phi(S, \omega) = \left( \int_{\xi_1} \omega, \dots, \int_{\xi_n} \omega
\right)$$ redefines local period coordinates with translation as change of
charts as above.\\

There is a natural action of $\GL(2,\R)$ on connected components of strata
coming from linear action of $\GL(2,\R)$ on $\R^2$ in charts. For any
translation surface in a stratum, its orbit closure via this action is some
affine invariant manifold of the stratum : it is defined in local period
coordinates by linear equations. They are endowed with a canonical measure
supported on these surfaces called affine measures \cite{EM}, \cite{EMM}.

\subsubsection{Translation cover}

To any primitive half-translation surfaces $S$ we associate its
\textit{translation cover} $\hat {S}$ corresponding to the subgroup of the
fundamental group with holonomy equal to $-1$. It is a double cover. We endow
$\hat S$ with the pulled-back metric of $S$ which defines a translation surface
structure for $\hat S$.\\

\noindent From a differential geometric point of view, we constructed a double cover of
$S$ on which the quadratic differential $q$ can be written $\omega^2$ where
$\omega$ is a holomorphic $1$-form.\\

Let $(S,q)$ be a half-translation surface in $\quadra(m_1, \dots m_d)$, $\hat
S$ its translation cover and $\hat \Sigma$ the preimage of its singular points.
Following \cite{AEZ}, assume there is a basis $\{ a_1, b_1, \dots, a_g, b_g \}$
of $H_1(S; \Z)$ which has trivial linear holonomy (this exists as long as there
is a zero with odd multiplicity) and let $\gamma_1, \dots, \gamma_{d-2}$ be
primitive non crossing elements of $H_1(S, \Sigma; \Z)$ representing a path
from $P_i$ to $P_{i+1}$ where $\{P_1, \dots, P_d\} = \Sigma$.\\

Given a saddle connection or an absolute cycle with trivial linear holonomy
$\gamma$, let $\gamma', \gamma''$ be its $2$ lifts in $\hat S$ endowed with the
orientation inherited from $\gamma$. Then we introduce $$\hat{\gamma} :=
\gamma' - \gamma''.$$
By definition $\hat{\gamma}$ belongs to $H^{-}_1(\hat S, \hat \Sigma;
\C)$ the $-1$-eigenspace of the linear automorphism induced by the deck
involution of the double cover.

\begin{proposition}
	The family $\{ \hat a_1, \hat b_1, \dots, \hat a_g, \hat b_g, \hat
	\gamma_1, \dots, \hat \gamma_{d-2} \}$ is a basis of $H^-_1(\hat S,
	\hat \Sigma; \C)$.
	\label{hatbasis} 
\end{proposition}

The dual family of
$\{ \hat a_1, \hat b_1, \dots, \hat a_g, \hat b_g, \hat \gamma_1,
\dots, \hat \gamma_{d-2} \}$ forms a basis of the anti-invariant
$1$-forms,
$$H^1_- (\hat S, \hat \Sigma; \C) \subset H^1(\hat S, \hat \Sigma; \C)$$
where the relative cohomology is a local chart for some abelian
stratum $\mathcal H (\alpha)$.  This period is twice the polygonal
periods we defined above up to a sign.\\

\noindent This is the basis we will be using to express equations of billiard
families.

\subsection{Windtree tables}
\label{windtree:flat}

Let $P$ a filled polygon which does not self-intersect will stand for the shape
of the obstacles in our infinite billiard.  Consider  the plane $\R^2$ on which
we place $P$ periodically centered at each point of a lattice $\Lambda$ as
scatterers such that copies do not overlap.  We denote 
the space consisting of the plane to which we removed
the inside of every obstacle by $\windtree(P, \Lambda)$.

\begin{definition}
	We call $\windtree(P, \Lambda)$ a $\Lambda$-periodic windtree table
	with obstacle $P$.
\end{definition}

Our purpose here is to understand the billiard flow on this infinite table and
its asymptotic speed. We denote the billiard flow by $$\phi_t^\theta:
\windtree(P, \Lambda) \mapsto \windtree(P, \Lambda).$$ For $p \in \windtree(P,
\Lambda)$ the point $\phi_t^\theta(p)$ is the position of the flow after time
$t$ starting from $p$ in direction $\theta$, which moves in straight lines
until it encounters an obstacle on which it bounces according to
Snell-Descartes law of reflection.

\begin{definition}
	In a windtree table $\windtree(P, \Lambda)$ for $d$ the euclidian
	distance on $\R$, $p \in \windtree(P, \Lambda)$ and $\theta \in [0,
	2\pi)$ the diffusion rate is the limit $$ \limsup_{t \to +\infty} \frac
	{\log d(p, \phi_t^\theta(p))} {\log t}.$$
\end{definition}

\subsubsection{Associated flat surface}

As the billiard table $\windtree(P, \Lambda)$ is $\Lambda$-periodic, we may
consider its quotient
$$\bigslant {\windtree(P, \Lambda)}{\Lambda} \simeq \bigslant{\R^2}{\Lambda} -
\mathring P =: \torusl(P)$$ which corresponds to playing billiard in a torus
with one copy of the obstacle $P$ placed in it.
Then we associate to it a flat surface on which the linear flow corresponds to
the billiard flow.\\

Take two copies of $\torusl(P)$ and glue the two copies of each side of $P$
using an isometry fixing the tangent vector and inverting the normal vector
(the axial symmetry along this side).  Now when the flow is bouncing in the
billiard, the geodesic flow of the flat surface is simply changing of copy in
the surface.\\

The gluing maps are changing orientation, hence in order for this surface to be
a flat surface as defined above we choose as a convention two opposite
orientations for the two copies. The change of charts now preserves orientation
and is in $\isoC$. We denote this flat surface by $S(P, \Lambda) = S$.\\

For each triple $(p, \theta, t) \in S \times [0, 2\pi) \times \R_+$ we define
an element $\gamma_t^\theta(p) \in H^1(S; \Z)$ as follows. Consider the
geodesic segment of lengths $t$ starting from $p$ in the direction $\theta$ and
close it by a small piece of curve that does not cross $h_\kappa$ and
$v_\kappa$. The curve used to close the segment can be chosen to be uniformly
bounded.\\

Let $h,v$ be a horizontal and vertical simple loop in $\torusl(P)$ that
generate the homology of the torus.  Let $h^S = h_1 - h_2$ and $v^S = v_1 -
v_2$ where $h_1, h_2$ (resp $v_1, v_2$) are the two lifts of $h$ (resp. $v$) in
$S$. And let a $f \in H^1(S; \Z^2)$ be a cocycle dual of $(h^S, v^S)$ with
respect to the intersection form.\\

The proposition below shows that the diffusion rate of a particle in a windtree
table $\windtree(P, \Lambda)$ can be reduced to the study of the pairing of the
approximate geodesic flow on $S$ with $f$.

\begin{proposition}[1 in \cite{DHL}]
	The diffusion rate of $\phi_t^\theta(\tilde p)$ is equal to
	$$\limsup_{t \to + \infty} \frac{\log \left| \left<f,
	\gamma_t^\theta(p)\right> \right|}{\log t}$$ when $p$ and $\tilde p$
	project to the same point on $\torusl(P)$.
\end{proposition}

\section{Lyapunov exponents}
\label{Lyapunov}

In the previous section, we have seen that the diffusion rate on a windtree
table is related to the asymptotic pairing of a cohomology class with a
modified linear flow on the associated translation surface. We will consider
throughout this section a translation $S$ in some abelian stratum $\mathcal
H(\alpha)$ which $\SL(2,\R)$-orbit closure is the affine invariant subspace
$\mathcal M \subset \mathcal H(\alpha)$ and $\nu_{\mathcal M}$ its invariant
measure.\\

The following theorem relates the diffusion rate with a Lyapunov exponent,

\begin{theorem*}[2 in \cite{DHL}]
	Let $F_1 \supset F_2 \supset \dots \supset F_{k}$ be the Oseledets flag
	decomposition of the Kontsevich-Zorich cocycle on $\mathcal M$, and
	$\lambda$ be its top Lyapunov exponent.  For every $\nu_{\mathcal
	M}$-Oseledets generic translation surface $S \in \mathcal H(\alpha)$,
	for every point $p \in S$ with infinite forward orbit, for all $f \in
	F_1 \setminus F_2$,
	$$\limsup_{t \to + \infty} \frac{\log |\left<f, \gamma_t(p)\right>|}{\log t} = \lambda.$$
\end{theorem*}

In \cite{EC} is proven that any translation surface
$S \in \mathcal H(\alpha)$ such that its $\SL(2,\R)$-orbit closure is
$\mathcal M$ is Oseledets generic in Lebesgue-almost every
direction. In particular they show the following theorem,

\begin{theorem}{1.5 in \cite{EC}}
	Fix $S \in \mathcal H_1(\alpha)$ and let $\mathcal M = \overline{\SL(2,\R)\cdot S}$ 
	the smallest affine invariant manifold containing $S$, 
	let $V$ be a $\SL(2,\R)$ invariant subbundle of the Hodge bundle which 
	is defined and continuous on $\mathcal M$. 
	Let $A_V : \SL(2,\R) \times \mathcal M \rightarrow V$ denote the restriction 
	of the Kontsevich-Zorich cocycle to $V$ and suppose that $A_V$ is strongly
	irreducible with respect to the affine measure $\nu_{\mathcal M}$ whose
	support is $\mathcal M$. Then, for almost every $\theta \in [0, 2\pi)$,
	$$\lim_{t \to \infty} \frac{\log || A_V(g_t r_\theta x)||}{\log t} \to \lambda_1$$ 
	where $\lambda_1$ is the top Lyapunov exponent of $A_V$.
\end{theorem}

A little modification in their argument, which we delay to the Annex,
enables us to show an additional lemma to this theorem.

\begin{lemma}
	In the previous theorem, for any $h \in V$ and almost every $\theta \in
	[0, 2\pi)$,
	$$ \lim_{t \to \infty} \frac{\log || A_V(g_t r_\theta x) h||}{\log t}
	\to \lambda_1.$$
	\label{genericity}
\end{lemma}

This reduces the computation of the diffusion rate of a windtree model to
determining irreducible components of the Kontsevich-Zorich cocycle along
$\SL(2,\R)$-orbits and in which of these is the cohomology class $f$.\\

In our case this will be done by the following irreducibility lemma,
\begin{lemma}
	In strata of quadratic differentials with at most simple poles,
	and more than 3 singularities that are not all of even order,
	the Kontsevich-Zorich cocycle is strongly irreducible on $H^+_1$
	for the action of $\SL(2, \R)$.
\end{lemma}

\begin{proof}
	The tautological bundle generated by the real and imaginary part of the
	abelian form associated to a surface in the stratum is contained in
	$H^-$ and not in $H^+$. Thus according to Theorem 1.1 of \cite{EFW},
	the algebraic hull of the Kontsevich-Zorich cocycle is the Zariski
	closure of monodromy. But the monodromy on $H^+$ is Zariski dense in
	$Sp(2g, \R)$ according to Section 6 in \cite{Rodolfo}. Hence $H^+$
	cannot have invariant subspaces for the Kontsevich Zorich cocycle, and
	is strongly irreducible.
\end{proof}

This implies the following,
\begin{corollary}
	Let $S$ be a half-translation surface which $\GL(2,\R)$ orbit is dense
	in a quadratic stratum, then for almost all direction and every point
	$p \in S$ with infinite forward orbit, for all $f \in H^1(S; \R)$,
	$$\limsup_{t \to + \infty} \frac{\log |\left<f, \gamma_t(p)\right>|}{\log t} = \lambda_1$$ 
	where $\lambda_1$ is the top Lyapunov exponent of the quadratic
	stratum.
	\label{diff_rate}
\end{corollary}


\section{Orbit closure}

\subsection{Some useful theorems}
\label{closure}

In this section we introduce some lemmas resulting from recent breakthrough in
the theory \cite{EMM}, \cite{Wright14} and \cite{Wright13}.

\begin{lemma}
	Let $\bilfam$ a family of flat surfaces in a fixed stratum which is
	represented in some period coordinates by a real linear subspace $B$. 
	Then for Lebesgue almost every $S \in \bilfam$, the orbit closure is an
	unique $\GL_2(\R)$-invariant suborbifold $\mathcal L$ of the
	Teichm\"uller space. Moreover, in the above period coordinate,
	$\mathcal L$ is a linear subspace $L$ such that $B \subset L$.
	\label{L}
\end{lemma}

This is the fundamental lemma in this article. Since it shows the
existence of one generic orbit closure which contains orbit
closures of all surfaces in the family. 

\begin{proof}
	According to \cite{EMM} Proposition 2.16 or alternatively
	\cite{Wright14} Corollary 1.9, there are countably many
	$\GL_2(\R)$-invariant closed orbifolds in each stratum.  Thus at least
	one orbit closure $\mathcal L$ of the family $\bilfam$ intersects
	$\bilfam$ with non-zero Lebesgue measure in $\bilfam$. In period
	coordinates, if two linear subspace $L$ and $B$ intersect with non-zero
	Lebesgue measure in $B$, then $B \subset L$.

	Take now $\mathcal L_0$ intersection of all $\mathcal L$ as above. This
	intersection, as any $\mathcal L$, is a closed $\GL_2(\R)$-invariant
	subset which contains $\bilfam$.  Thus the orbit closure of any point
	of $\mathcal L_0$ is contained in $\mathcal L_0$. This implies that any
	$\mathcal L$ as above coincide with $\mathcal L_0$. Thus for any $S \in
	\bilfam$ which orbit closure has non-zero measure intersection with
	$\bilfam$, $\overline{\GL_2(\R) \cdot S} = \mathcal L_0$.

	Hence for any $S \in \bilfam$ such that $\mathcal N :=
	\overline{\GL_2(\R) \cdot S} \neq \mathcal L_0$, Lebesgue measure of
	$\mathcal N \cap \bilfam$ is zero.  Taking out the countably many such
	subset of $\bilfam$, the set of remaining points is of full Lebesgue
	measure and the orbit closure of each of these points is $\mathcal
	L_0$.
\end{proof}

\begin{lemma}
	Suppose with the notation of the previous lemma, that $B$ contains a
	$\R$-linear subspace $D$.  Let $D_{Re}$ and $D_{Im}$ be the projections
	of $D$ to $H^1(S; \R)$ and $H^1(S; i\R)$.  Then $\operatorname{Vect}_\C (D_{Re},
	D_{Im}) \subset L$ \textit{i.e.}  $L$ contains the $\C$-linear span of
	$D_{Re}$ and $D_{Im}$.
	\label{real} 
\end{lemma}

This lemma will enable us to show that restrictions on the directions of the sides of
obstacles do not interfere with the orbit closure.

\begin{proof}
  By Lemma \ref{L}, we know that $D \subset L$. By \cite{Wright14}
  the field of definition of such an affine manifold is
  real, in particular it is the complexification of $L_{Re}$.
\end{proof}

\begin{lemma}
	Let $S$ be a half-translation surface, $\Sigma$ the set of its
	singularities, and $\gamma_1, \gamma_2, \dots, \gamma_d$ a basis of
	primitive non-crossing elements of $H_1(S,\Sigma; \Z)$. We denote by
	$\hat \gamma_i$ their periods in $S$.\\ 

	If $\eta$ is the homology of the union
	of core curves of $\mathcal L$-parallel cylinders in the surface
	associated to periods $\hat \gamma_i$. Then for all $\delta$ in a
	neighborhood of zero in	$\mathbb C$, the surface with periods 
	$$\hat \gamma_i + \left< \eta, \gamma_i \right> \delta$$ is in the orbit closure $\mathcal L$.
	\label{cylinder_deformation}
\end{lemma}

Once we have conjectured what the generic orbit closure of our billiard should
be, our goal will consist in finding a surface in the family that has some good
cylinder decomposition. Using this lemma there will be surfaces in the orbit
closure breaking some symmetry, and by induction we will show density.

\begin{proof}
	This is a direct corollary of Lemma 4.11 in \cite{Wright13}.
	Each cylinder deformation adds some complex number $\delta$
	to the period of a given path for every intersection.
	By shearing and stretching any $\delta$ can be attained in
	a neighborhood of zero.
\end{proof}

\subsection{Periodic windtree with several obstacles}

Choose a layout for $n$ rectangular obstacles in the plane, all oriented in the
same horizontal direction. Now repeat $\Z^2$-periodically this pattern in the
plane, assuming at an initial step that they do not overlap. In other terms,
pick a square torus in which you place $n$ horizontal rectangular
obstacles. We call $\mathcal B_n$ this family of billiards.  We investigate its
generic $\GL(2,\R)$ orbit closure to compute its diffusion rate.
The case of $\mathcal B_1$ was done in \cite{DHL} in which the authors proved
that the diffusion rate is $2/3$. \\

As in \ref{windtree:flat} we associate to each windtree table in $\mathcal B_n$
a half-translation surface in $\quadra(1^{4n})$. This yields an embedding
$$\mathcal S: \mathcal B_n \mapsto \quadra(1^{4n}) $$\\

\subsubsection{Orbit closure} We prove the following Lemma,
\begin{lemma}
	For Lebesgue almost every windtree table in $B_n$, $n \geq 2$,  the
	image of its associated half-translation surfaces in $\quadra(1^{4n})$
	has a dense $\GL(2,\R)$-orbit.
\end{lemma}

Let $S \in \mathcal S (\mathcal B_n)$, the surface $S$ has genus $n+1$ and the
stratum $\quadra(1^{4n})$ has dimension $6n$. We consider $a_1, b_1, a_2, b_2$
simple loops which generate the homology of the two copies of the torus, and
take $c_1, d_1, \dots, c_{n-1}, d_{n-1}$ the loops around the obstacles and
between two consecutive obstacles. These generate the absolute homology of $S$.
Now for each obstacle $i$, start at the lower left corner and browse the rectangle
clockwise, we denote by $\alpha_i, \beta_i, \alpha_i'$ the three saddle
connection we cover until the lower right corner. Let $\gamma_i$ be a path from
the lower right corner of obstacle $i$ to lower left corner of obstacle $i+1$.

\begin{figure}[h]
  \centering
  \includegraphics[scale=.35]{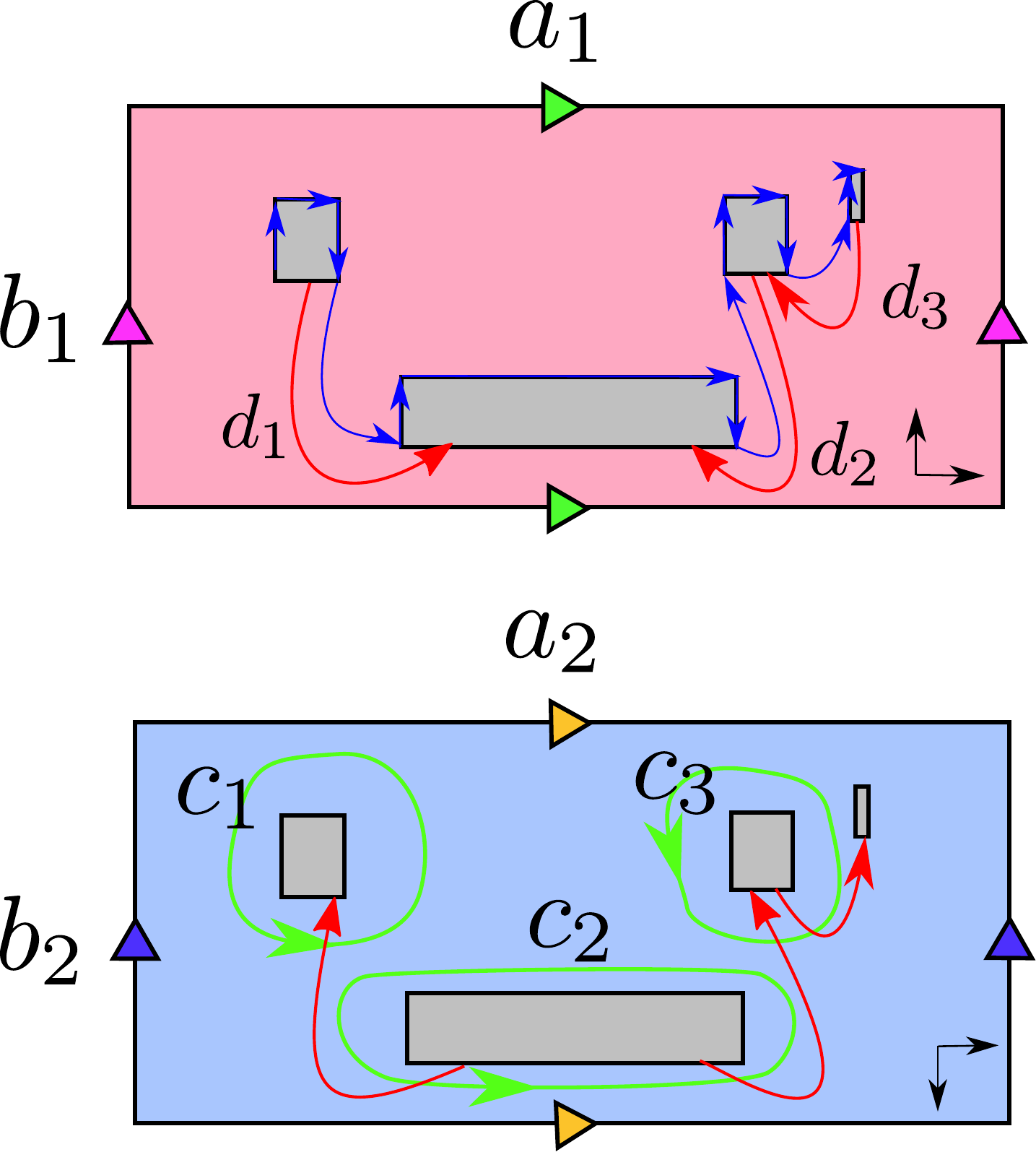}
  \caption{Basis for relative homology.}
  \label{rel_basis}
\end{figure}

These paths form a basis of the relative homology group of $S$. According to
Lemma \ref{hatbasis} if we take the hat image of these homology elements
besides from $\hat \alpha_{n-1}'$ they form a basis of $H^-_1(\hat S, \hat
\Sigma; \C)$ which induce local coordinates in the stratum (see e.g.
\cite{AEZ}) called period coordinates.

We also introduce $\beta'_i$ for $1 \leq i \leq n-1$ the last side of the
rectangle that closes obstacle $i$. In other term, the class that satisfies
$\alpha_i + \beta_i + \alpha'_i + \beta'_i = c_i$. For ways of intersection
numbers with cylinders we will construct later, we will prefer to replace $c_i$
by $\beta'_i$ in the basis and equations. 

To write down equations in period coordinates we need to eliminate an ambiguity
given by the non trivial holonomy of the surface. We choose a fundamental domain
for the action of this holonomy given by the two copies glued along the
vertical sides to which we remove the horizontal sides. This corresponds to
drawing the copies reflected along the horizontal axis. Now the family 
$\mathcal S$ is defined locally by the following equations, where we make the
abuse to write the homology class while meaning their period,

\begin{align}
  \label{eq:first}
  \begin{split}
    & \im(a_1) = 0 \text{, } \re(b_1) = 0 \text{, } \re(a_1) = \im(b_1)\\
    & \im(\beta_i) = 0 \text{, } \re(\alpha_i) = 0 \text{ for all } 1 \leq i \leq n\\
    & d_i = \gamma_i - \overline \gamma_i \text{ for all } 1 \leq i
    \leq n-1
  \end{split}\\
  \nonumber \vspace{1cm}\\
  \label{eq:second}
  \begin{split}
    & a_1 = a_2 \text{, } b_1 = -b_2\\
    & \alpha_i = -\alpha_i' \text{ for all } 1 \leq i \leq n-1\\
    & \beta_i = -\beta_i' \text{ for all } 1 \leq i \leq n-1\\
  \end{split}\\
  \nonumber \vspace{1cm}
\end{align}

There are $2n+3$ real equations and $3n-1$ complex equations.  The quadratic
stratum is of complex dimension $6n$, thus the induced subspace is of real
dimension $12n -2n -6n -1 = 4n - 1$.  On the other hand for the family of
billiards, we have $2n$ variables for the size of each obstacle, $2n-2$ for
relative position of the obstacles, and $1$ dimensions for the size of the
square torus. Thus we have indeed listed all the equations that define our
billiard family.\\

Below we show that these two sets of equations do not constrain the generic
orbit closure for our billiards which as a consequence will be the whole
stratum.  The first argument relies on Lemma \ref{real} and the second on Lemma
\ref{cylinder_deformation}.\\

First remark that the periods appearing in equations $(\ref{eq:first})$  are
not constrained by equations $(\ref{eq:second})$.  Lemma \ref{real} then
implies that the affine space corresponding to the orbit closure $L$ contains
$\operatorname{Vect}_\C \left(a_1, \beta_i, d_i, \gamma_i\right)$ and $\operatorname{Vect}_\C \left(b_1,
\alpha_i, \gamma_i\right)$ and consequently does not satisfy neither of the
equations in (\ref{eq:first}). We have shown that the orbit closure contains the
space defined by equations (\ref{eq:second}). \\

In the following we demonstrate inductively that $L$ contains affine spaces
defined by a smaller subset of equations in (\ref{eq:second}) which will eventually be
empty. To do so we point out surfaces in the space defined by the given
subset which have a cylinder rationally independent to any other ones in the
same direction. We will show that all but one equations of this subset are
respected by the shifted periods in Lemma \ref{cylinder_deformation}. This
will imply that the orbit closure contains the subspace defined by all but this
latter equation.\\

We want to decorrelate the periods of $b_1$ and $b_2$ but in the family a
cylinder in the torus along $a_1$ has always a symmetric counter-part along
$a_2$. We use the fact proven above that in the orbit closure $\gamma_i$ and
$d_i$ have no correlation thus we can move the obstacles in the two copies
independently. Figure \ref{cylinder2} shows how to have a cylinder in one torus
and not in the other by moving the obstacles and obstructing the flow in one
copy. For a generic choice of lengths, the hatched cylinder is not
commensurable to any other cylinder and its core curve intersects only $b_1$.
The cylinder deformation breaks the relation between $a_1$ and $a_2$ and the
same construction in the vertical direction  breaks the relation between $b_1$
and $b_2$. As a result, the affine space $L$ contains the space defined by
equations (\ref{eq:second}) minus the equations on $a$ and $b$.

\begin{figure}[h!]
  \centering
  \includegraphics[scale=.4]{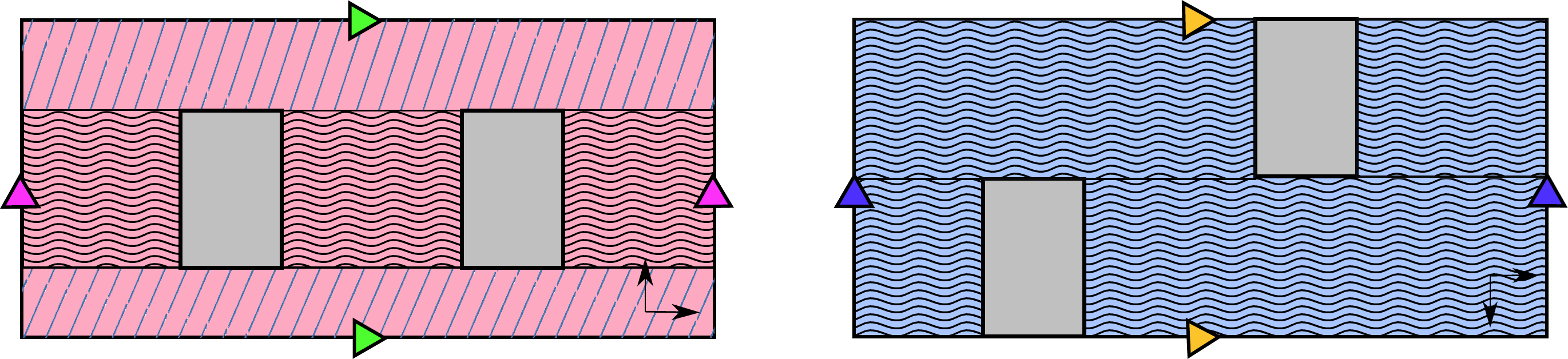}
  \caption{Good cylinder decomposition to deform $b_1$}
  \label{cylinder2}
\end{figure}

Consider now the billiard with the same square obstacles of irrational side
length such that all the obstacles are aligned in order. The distances between
the obstacles are chosen such that they are rationally independent. On these
surfaces there is a full decomposition in cylinders and all of the cylinders
are rationally independent. The cylinder going from the right of the last
obstacle to the left of the first intersects $b_1, \alpha_1$ and $b_2$.  The
number of intersection of the core curve with each one of these curves is one.
The previous argument has eliminated the constrains on $b_1$ and $b_2$ thus
this cylinder deformation breaks the relation between $\alpha_1$ and
$\alpha'_1$.

Now by induction we take the cylinder intersecting $\alpha'_i$ and
$\alpha_{i+1}$. By assumption $\alpha'_i$ does not appear in any equation and
so we can break the relation between $\alpha_{i+1}$ and $\alpha'_{i+1}$.

\begin{figure}[h!]
  \centering
  \includegraphics[scale=.4]{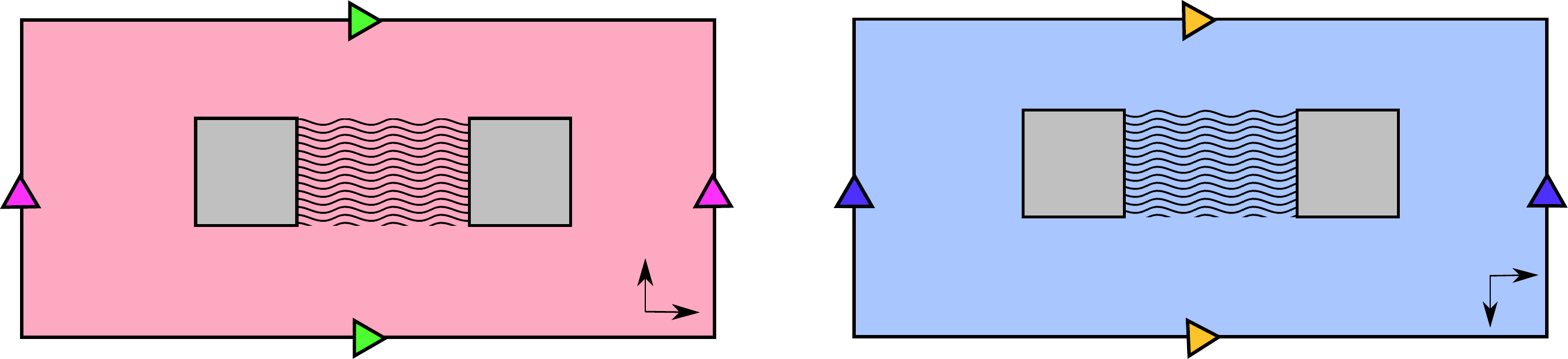}
  \caption{Good cylinder decomposition to deform $\alpha_i$}
  \label{cylinder}
\end{figure}

The same argument can be applied in the vertical direction for $\beta_i$ and
$\beta'_i$. This ends the proof of generic density for billiards in $\mathcal S
(\mathcal B_n)$.\\

This density result together with Lemma \ref{L} and
Corollary \ref{diff_rate} imply the following,
\begin{theorem}
  The diffusion rate for Lebesgue-almost every windtree model in 
  $B_n$ with $n \geq 2$ in Lebesgue-almost every direction is equal
  to the top Lyapunov exponent of $\quadra(1^{4n})$.
\end{theorem}

\subsubsection{Obstacles with many right angles.} Consider now a more general
periodic windtree table with $n$ obstacles which are horizontal polygons with
right angles. For each obstacle $i$ there are $k_i$ inward (concave) and $4 +
k_i$ outward (concave) right angles. Which implies that the obstacle has
$2+k_i$ vertical and $2+k_i$ horizontal sides.
We denote this family by $\mathcal B_n(k_1, \dots, k_n)$.\\

The associated quadratic differential has simple zeros at the outward
right angles and poles at the inward. It has genus $n+1$ and is
in the stratum $$\quadra(1^{4n + p}, -1^{p})$$ where $p = \sum k_i$.\\

To construct a basis of homology of the associated translation surface, we
start from the left point of the lowest horizontal side and browse the obstacle
boundary clockwise until we come back to the starting point. This yields saddle
connections $\alpha^1_i, \beta^1_i, \alpha^2_i, \dots, \alpha^{4+k_i}_i,
\beta^{4+k_i}_i$. The classes $\alpha^{4+k_n}_n$ and $\beta^{4+k_n}_n$ are not
taken into consideration to yield a basis of $H^-_1(S, \Sigma; \C)$. Let
$\gamma_i$ be the path joining the starting points two consecutive obstacles
$i$ and $i+1$ and  define as in the previous section absolute homology classes
$a, b$ and $d$.\\

The equations in period coordinates are very similar as in the previous case,
we only need to adapt equations on the obstacles.\\
\begin{align*}
	\begin{split}
		& \im(a_1) = 0 \text{, } \re(b_1) = 0 \text{, } \re(a_1) = \im(b_1)\\
		&\re(\alpha^j_i) = 0 \text{ for all } 1 \leq i \leq n \text{ and } 1 \leq j \leq 2+k_i-1\\
		&\im(\beta^j_i) = 0 \text{ for all } 1 \leq i \leq n \text{ and } 1 \leq j \leq 2+k_i-1\\
		& d_i = \gamma_i - \overline \gamma_i \text{ for all } 1 \leq i
		\leq n-1
	\end{split}\\
	\nonumber \vspace{1cm}\\
	\begin{split}
		& a_1 = a_2 \text{, } b_1 = -b_2\\
		&\sum_{j=1}^{4+k_i} \alpha^j_i = 0 \text{ for all } 1 \leq i \leq n-1\\
		&\sum_{j=1}^{4+k_i} \beta^j_i = 0 \text{ for all } 1 \leq i \leq n-1\\
	\end{split}\\
	\nonumber \vspace{1cm}
\end{align*}
There are now $\sum (4 + 2k_i-2) + 3 = 2n + 2p + 3$ real equations and $3n-1$
complex equations.  The quadratic stratum is of complex dimension $$2(n+1) + 4n
+ 2p -2 = 6n + 2p$$ thus the induced subspace is of real dimension $$12n + 4p -
2n - 2p -3 - 6n + 2 = 4n + 2p - 1.$$ 
On the other hand for the family of billiards, we have $\sum (4+2k_i -2) = 2n +
2p$ variables for the size of each obstacle, $2n-2$ for relative position of
the obstacles, and $1$ dimensions for the size of the square torus. Thus we
have indeed listed all the equations that define our billiard family.\\

The first part of the previous argument applies \textit{verbatim} to this case
with the real and imaginary part equations. For the second part we need to
exhibit a similar construction of cylinders. The construction of Figure
\ref{cylinder2} is straightforward to generalise to any shape of
obstacle. We will detail the generalisation of the construction in Figure
\ref{cylinder}. 

Start with the vertical side that does not appear in the
basis. Now we can find an element of the family such that the obstacle $n$ is
in the neighborhood of a rectangle as in Figure \ref{deformation}, making every
other side very small, and similarly for the first obstacle. There is a
horizontal  cylinder joining the given side of obstacle $n$ with a side of the
first obstacle. This surface will be completely decomposed into horizontal
cylinders and the lengths are chosen to be all rationally independent.

\begin{figure}[h]
	\centering
	\begin{subfigure}{.45\textwidth}
		\centering
		\includegraphics[scale=.4]{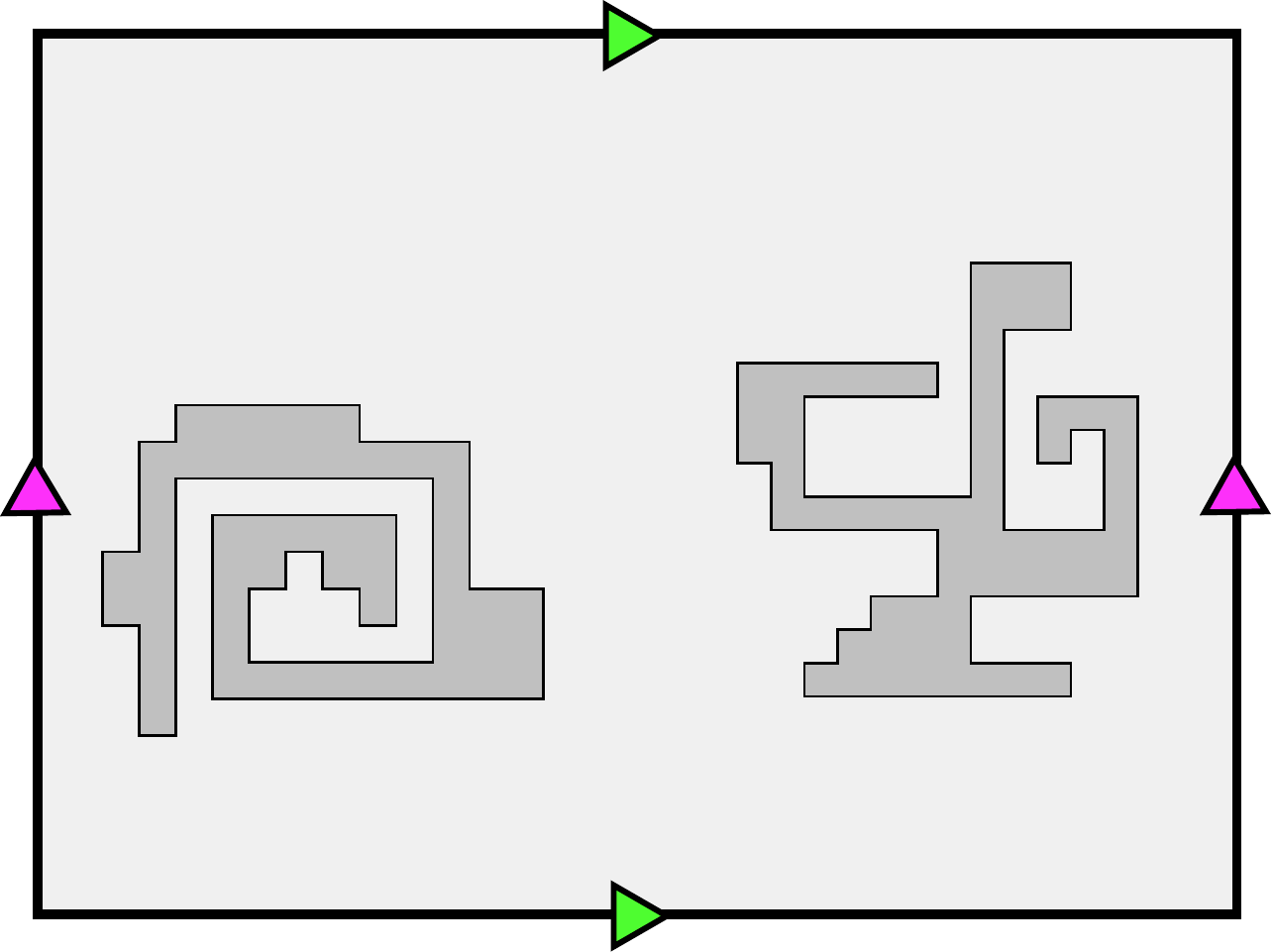} 
	\end{subfigure}
	\begin{subfigure}{.45\textwidth}
		\centering
		\includegraphics[scale=.4]{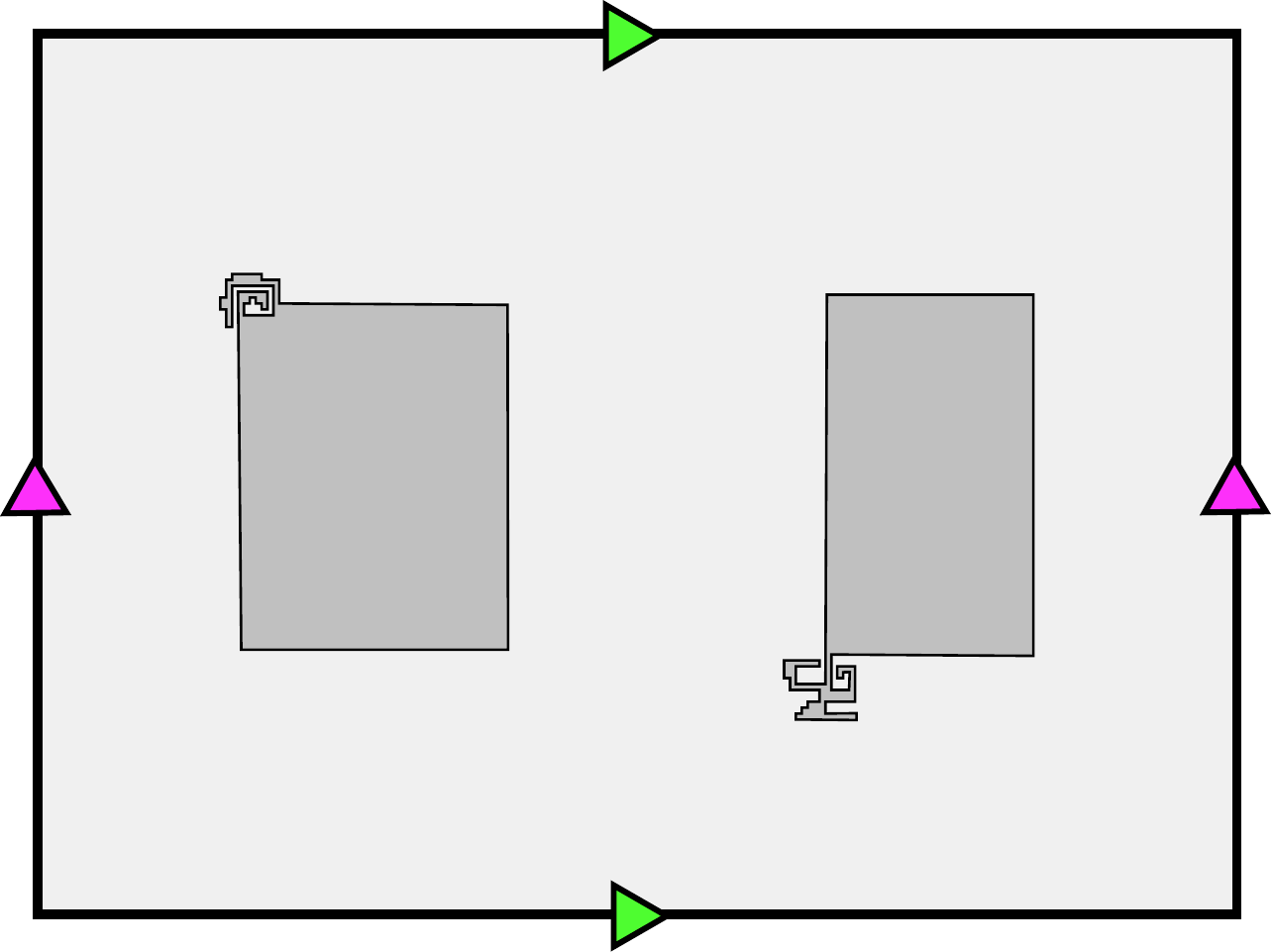} 
	\end{subfigure}
	\caption{Example of deformation}
	\label{deformation}
\end{figure}

This enables us to break the equation constraining the $\alpha_1^j$.
Then by induction we show that a generic billiard in
$\mathcal B_n(k_1, \dots, k_n)$ induces a quadratic
differentials with dense orbits in the stratum. We have the
following theorem,

\begin{theorem}
	For any $n \geq 2$,  $k_1, \dots k_n \geq 0$ and $p = \sum k_i$,
	the diffusion rate for in Lebesgue-almost every windtree model
	$B_n(k_1, \dots, k_n)$ in Lebesgue-almost every direction is equal to
	the top Lyapunov exponent of $\quadra(1^{4n + p}, -1^p)$.
\end{theorem}

\section{Some numerical computations}

Figure \ref{no_pole} shows numerical approximations of the principal Lyapunov
exponent of strata $\quadra (1^{4n})$. We observe that it goes to $1/2$ when $n
\to \infty$.

\begin{figure}[h!]
  \centering
  \includegraphics[width=8cm]{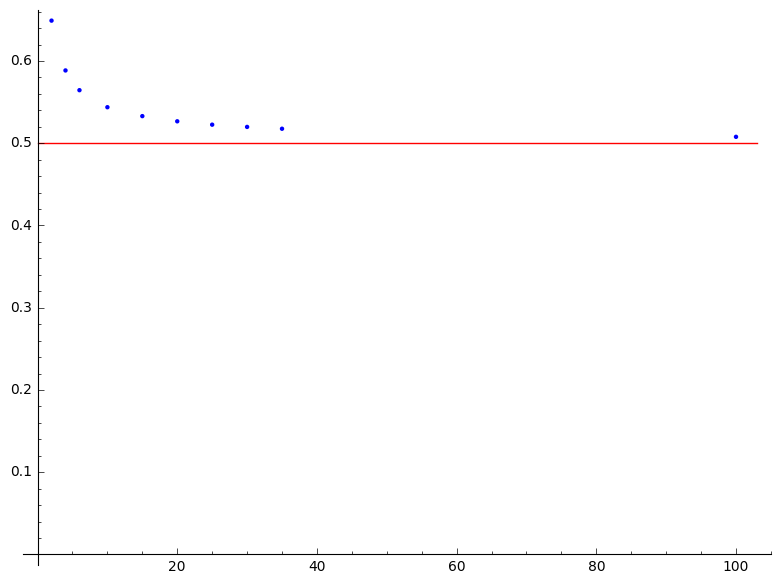}
  \caption{Principal Lyapunov exponent for $\quadra(1^{4n})$}
  \label{no_pole}
\end{figure}

In Figure \ref{ten_pole}, we represent a computation of the principal
Lyapunov exponent for $\quadra(1^{4n + 10}, -1^{10})$. When we
fix the number of simple poles and increase the number of simple zeros, the
diffusion rate again goes to $1/2$ but now by smaller values.\\

\begin{figure}[h!]
  \centering
  \includegraphics[width=8cm]{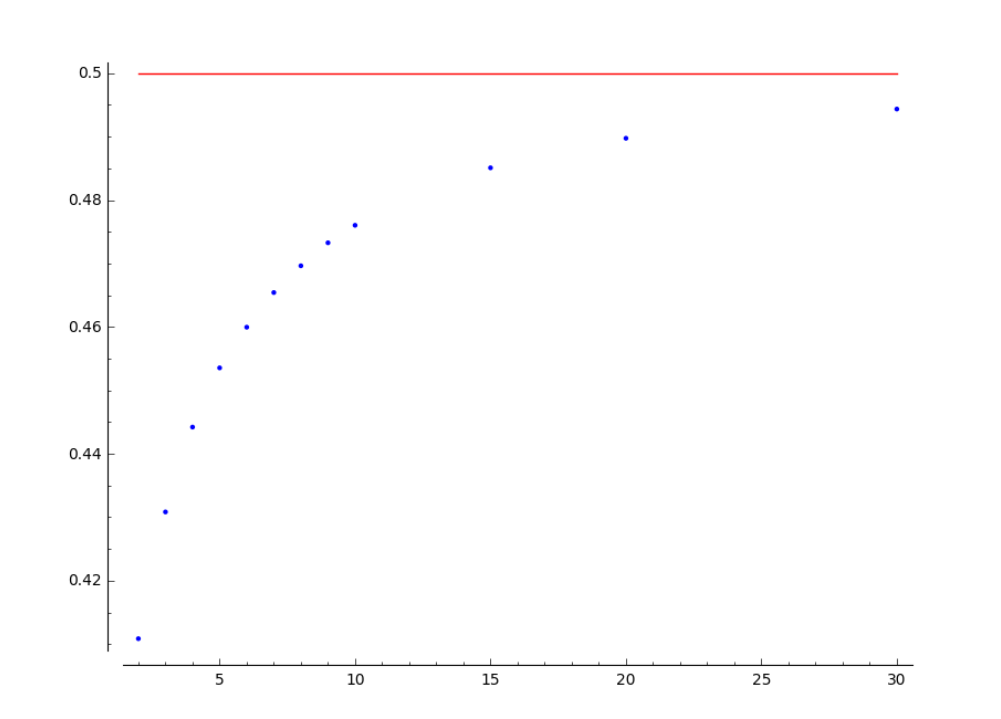}
  \caption{Principal Lyapunov exponent for
    $\quadra(1^{4n+10}, -1^{10})$}
  \label{ten_pole}
\end{figure}

The $1/2$ value is also the diffusion rate for the Brownian motion.
Intuitively, these convex angles scatter the linear flow which follows
completely different paths from one side to the other of the singularity.
They mimic the hyperbolic behaviour of smooth convex obstacles.\\

An opposite behaviour is given by the concave right angles of the
obstacles. In Figure \ref{poles}, we present the largest Lyapunov
exponent of strata corresponding to windtrees with two obstacles
with an increasing number of concave angles.

\begin{figure}[h!]
  \centering
  \includegraphics[width=8cm]{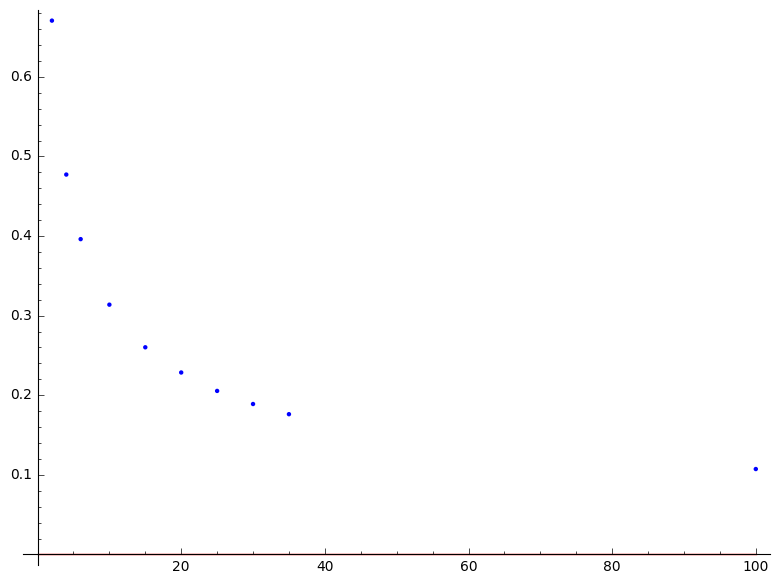}
  \caption{Principal Lyapunov exponent for $\quadra(1^{8+p}, -1^{p})$}
  \label{poles}
\end{figure}

Further experiments show that in contrary to the previous case for  a fixed
number of simple zeros and a number of simple poles going to infinity, the
principal Lyapunov exponent is going to zero. A heuristic explanation for this
phenomenon is that when the flow hits the obstacle close to a concave right
angle in the billiard it comes back on its steps slightly shifted as drawn in
Figure \ref{concave}. This enters in resonance with the result of \cite{DZ}
which states that when we increase the number of concave right angles of a
single obstacle for a periodic windtree, the diffusion rate goes to zero.  This
also enters in the frame of the more general Grivaux-Hubert conjecture that we
explore and reformulate in \cite{Fouge}. \\

\begin{figure}[h]
  \centering
  \includegraphics[width=5cm]{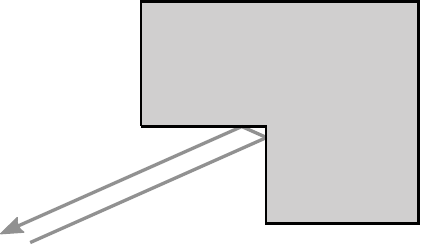}
  \caption{Flow bouncing close to a concave right angle.}
  \label{concave}
\end{figure}

\clearpage
\appendix
\section{Generic Lyapunov exponent}
In this section we follow the proof of \cite{EC} which 
shows that any translation surface is Lyapunov and Birkhoff generic in
its orbit closure for almost every direction. We will focus on one of
the key results in this article about Lyapunov genericity on a
irreducible component for the Kontsevich-Zorich cocycle.

\begin{theorem*}[1.5 in \cite{EC}]
  Fix $x \in \mathcal H_1(\alpha)$ and let
  $\mathcal M = \overline{\SL(2,\R) x}$ the smallest affine invariant
  manifold containing $x$, let $V$ be a $\SL(2,\R)$ invariant
  subbundle of the Hodge bundle which is defined and continuous on
  $\mathcal M$. Let $A_V : \SL(2,\R) \times \mathcal M \rightarrow V$
  denote the restriction of the Kontsevich-Zorich cocycle to $V$ and
  suppose that $A_V$ is strongly irreducible with respect to the
  affine measure $\nu_{\mathcal M}$ whose support is $\mathcal M$.
  Then, for almost every $\theta \in [0, 2\pi)$,
  $$ \lim_{t \to \infty} \frac{\log || A_V(g_t r_\theta x)||}{\log t} \to \lambda_1 $$
  where $\lambda_1$ is the top Lyapunov exponent of $A_V$.
\end{theorem*}

\noindent Our purpose here is to show the following additional lemma
to this theorem, introduced as Lemma \ref{genericity} in section \ref{Lyapunov}.

\begin{lemma*}
  In the previous theorem, for any $h \in V$ and almost every
  $\theta \in [0, 2\pi)$
  $$ \lim_{t \to \infty} \frac{\log || A_V(g_t r_\theta x) h||}{\log t} \to \lambda_1.$$
\end{lemma*}

In \cite{EC} intuition of the result is provided by analogy with
random walks. We start by showing the analog of Lemma \ref{genericity}
for random walks.

\subsection{Random walks}

Let $\mu$ be a $\SO(2,\R)$-invariant compactly supported measure on $\SL(2,\R)$
which is absolutely continuous with respect to Haar measure. A measure $\nu$ on
$\mathcal H_1(\alpha)$ is called $\mu$-stationary if $$\mu * \nu = 
\int_{\SL(2,\R)} (g_* \nu) d\mu(g) = \nu.$$ By a theorem of Furstenberg
\cite{FU1}, \cite{FU2}, restated in \cite{NZ}[Theorem 1.4], there exists a
probability measure $\rho$ on $\SL(2,\R)$ such that the map $\nu \to \rho *
\nu$ is a bijection between ergodic measures for the action of upper triangular
subgroup of $\SL(2,\R)$ and ergodic $\mu$ stationary measures which are
$\SL(2,\R)$-invariant affine measures according to \cite{EM}[Theorem 1.4].\\
This is a first step for an analogy between Teichmüller flow in some affine
invariant locus and a random walk with the associated measure.\\

Let $\Gr_s$ denote the grassmanian of $s$-dimensional subspaces in the
$\SL(2,\R)$ invariant subbundle of the Hodge bundle $V$. Let $\tilde {\mathcal
H} = \mathcal H_1(\alpha) \times \Gr_s$ and $\tilde \nu$ be the $\mu$
stationary measure on it; we may write $d \tilde \nu (x, U) = d \nu(x) d
\eta_x(U)$.

The measure $\eta_x$ on $\Gr_s$ heuristically corresponds to the mean position
of any linear subspace carried along the Teichmüller flow using Gauss-Manin
connection.  Let $h$ be some vector in $V \setminus 0$ and $I(h) \subset
\Gr_s$ be the set of $s$-dimensional subspaces containing $h$.

\begin{lemma*}[C.10 in \cite{EM}]
	If the cocycle $A_V$ is strongly irreducible on $V$ then for almost
	every $x \in \mathcal H_1(\alpha)$ and any vector $h_x \in V$,
	$\delta_x(I(h_x)) = 0$
\end{lemma*}
In particular if we consider some Oseledets flag this Lemma
yields that generically they do not contain a fixed vector $h$ along
random walks.\\

We show a random walk version of the theorem in the previous paragraph,
\begin{theorem}[Theorem 2.6 and Lemma 2.9 of \cite{EC}]
	Fix $x \in \mathcal H_1(\alpha)$ and let $\mathcal M =
	\overline{\SL(2,\R) \cdot x}$ the smallest affine invariant manifold
	containing $x$, let $V$ be a $\SL(2,\R)$ invariant subbundle of the
	Hodge bundle which is defined and continuous on $\mathcal M$. Let $A_V
	: \SL(2,\R) \times \mathcal M \rightarrow V$ denote the restriction of
	the Kontsevich-Zorich cocycle to $V$ and suppose that $A_V$ is strongly
	irreducible with respect to the affine measure $\nu_{\mathcal M}$ whose
	support is $\mathcal M$. Then for a fixed $h \in V$ and for
	$\mu^\N$-almost every $\overline g = (g_1, \dots, g_n, \dots)$, $$
	\lim_{n \to \infty} \frac 1 n \log || A_V(g_n \dots g_1, x) h|| \to
	\lambda_1 $$ where $\lambda_1$ is the top Lyapunov exponent of $A_V$.
	\label{thm:random}
\end{theorem}

This theorem already appears in \cite{EC} as a remark to a more general
theorem. We reformulate the proof in this specific case for convenience to the reader.

\subsection{Proof Theorem \ref{thm:random}}

We fix $\mathcal M$ and $V$ as in the theorem.  Pick an arbitrary $v_0 \in V$
and let $v_i(\overline g) = A_V(g_i \dots g_1, x)v_0$.
The key tool to show this theorem is a decomposition lemma for the sequences of
cocycle in the case of strong irreducibility.

\begin{lemma*}[2.11 and 2.16 in \cite{EC}]
	For all $\epsilon > 0$, there exists an integer $L$ such that for every
	$x \in \mathcal M$ almost every $\overline g$ we have that all but a
	set of $\N$ of density $4 \epsilon$ is in disjoint blocks $[i+1, i+L]$
	so that
	$$\exp (\lambda_1 - \epsilon)^L \leq \frac {|| A_v (g_{i+L} \dots
	g_{i+1}, y) v ||}{||v||} \leq \exp (\lambda_1 + \epsilon)^L.$$
\end{lemma*}
\begin{proof}
	Refer to section 2.3 of \cite{EC}.\\
\end{proof}

Now let $\overline g$ be in the full measure set as above, $K$ be the subset of
density $4 \epsilon$ and $I$ the set of indices $i$ in the blocks $[i+1, i+L]$.
Then for $n \gg L$,
\begin{align*}
	\log || v_n || &= \sum_{i=1}^n \log \frac {||v_{i}||}{||v_{i-1}||}\\
		       &= \underbrace{\sum_{i \in I \cap [1, n-L]} \log \frac
	{||v_{i+L}||}{||v_{i}||}}_{S_1} + \underbrace{\sum_{i \in K \cap [1,
n-L]} \log \frac {||v_{i}||}{||v_{i-1}||}}_{S_2} +
\underbrace{\sum_{i=n-L'+1}^n \log \frac
{||v_{i}||}{||v_{i-1}||}}_{S_3}
\end{align*}
where $n-L' = \max \{ n-L, I + L \}$.\\

Let $C$ such that for all $g$ in the support of $\mu$ and all $y \in \mathcal
M$, $|| A_V(g,y) || \leq C$. Then $|S_3| \leq L \log C$, and $|S_2| \leq 4
\epsilon n \log C$.\\ 
Moreover
$$\frac{|I \cap [1, \dots, n]| \cdot L}{n} \geq 1 - 4\epsilon$$
Hence
$$S_1 \geq \left|I \cap [1, \dots, n]\right| \cdot (\lambda_1 -
\epsilon)L \geq (1-4\epsilon)n (\lambda_1 - \epsilon)$$
and
$$\frac 1 n \log ||v_n|| \geq (1-4\epsilon)(\lambda_1 - \epsilon) - 4\epsilon \log C
- \frac L n \log C$$ 
for almost every $\overline g$ and any $n \gg L$.\\

Since $\epsilon \geq 0$ is arbitrary, we get for
all $h\in V$ and almost every $\overline g$,
$$\liminf_{n\to\infty} \log || A_V(g_n \dots g_1, x) h|| \geq \lambda_1.$$
And with a similar argument we get an upper bound
$$\limsup_{n\to\infty} \log || A_V(g_n \dots g_1, x) h|| \leq \lambda_1.$$
Which implies Theorem \ref{thm:random}.

\subsection{Proof of Lemma \ref{genericity}}

According to the sublinear tracking Lemma of \cite{EC}, for almost every
$\theta \in [0, 2\pi)$, there exists $\overline g = (g_1, \dots, g_n, \dots)$
satisfying Theorem \ref{thm:random} such that we can write
$$g_{\lambda n} r_\theta =  \epsilon_n g_n \dots g_1$$
with $\epsilon_n \in \SL(2,\R)$ satisfying
$$ \lim_{n\to\infty} \frac 1 n \log ||\epsilon_n|| = 0$$
By the cocycle relation we have
$$A_V (g_{\lambda n}, r_\theta x) = A_V(\epsilon_n, g_n \dots g_1 x) A_V(g_n \dots g_1, x).$$
But there exists $C > 0$ and $N < \infty$ so that for all
$g \in \SL(2, \R)$ and all $x \in \mathcal H_1(\alpha)$,
$$|| A_V(g,x) || \leq C ||g||^N.$$

\noindent Hence
$$\log || A_V(g_{\lambda n}, r_\theta x) h || = \log || A_V(g_n \dots g_1, x)|| h + o(n).$$
Which shows the Lemma.\\

\newpage
\bibliographystyle{alpha} \bibliography{intro}
\end{document}